\documentclass{amsart}
\usepackage{amsmath}
\usepackage{amssymb,url}
\usepackage{epsfig}
\usepackage{hyperref}

\numberwithin{equation}{section}

\title[MULTIPLICATIVE DIOPHANTINE EXPONENTS OF HYPERPLANES]{MULTIPLICATIVE DIOPHANTINE EXPONENTS OF HYPERPLANES AND THEIR NONDEGENERATE SUBMANIFOLDS}

\author{Yuqing Zhang}

\address{Erwin Schroedinger Institute,
Boltzmanngasse 9,
A-1090 Vienna, Austria}

\email{\url{zhangy6@univie.ac.at}}

\newcommand{\R}{{\mathbb{R}}}

\newcommand{\Z}{{\mathbb{Z}}}

\newcommand{\N}{{\mathbb{N}}}

\newcommand{\q}{{\bf{q}}}

\newcommand{\GL}{\operatorname{GL}}

\newcommand{\SL}{\operatorname{SL}}

\newcommand{\diag}{{\rm diag}}

\newcommand{\ou}{{\omega(\y)}}

\newcommand{\ot}{{\omega^\times(\y)}}

\newcommand{\LL}{{\mathcal L}}

\newcommand{\x}{{\bf x}}

\newcommand{\z}{{\bf z}}
\newcommand{\w}{{\bf w}}

\newcommand{\p}{{\bf p}}
\newcommand{\y}{{\bf y}}

\newcommand{\f}{{\mathbf{f}}}

\newcommand{\MM}{{\mathcal{M}}}

\newtheorem{theorem}{Theorem}[section]

\newtheorem{lem}[theorem]{Lemma}

\newtheorem{prop}[theorem]{Proposition}

\newtheorem{condition}[theorem]{Condition}

\newtheorem{remark}[theorem]{Remark}

\thanks{The author is supported by
Austrian Science Fund (FWF) Grant NFN S9613.}

\begin{document}

\begin{abstract}

We consider multiparameter dynamics on
the space of unimolular lattices. Along with quantitative
nondivergence we prove that multiplicative Diophantine exponents of hyperplanes are inherited  by their nondegenerate
submanifolds.

\end{abstract}

\maketitle

\section{Introduction}

Given any $\y=(y_1,\ldots,y_n) \in \R^n$ , we define its Diophantine
exponent as
\begin{equation}\label{eq: omega}
\omega(\y)=\sup \left\{v\mid \exists \;\infty
\text{ many } \q \in \Z^n  \text{ with } |\langle\q,\y\rangle+p|<\|\q\|^{-v}
\text{ for some } p \in \Z\right\},
\end{equation}
where $\langle\q,\y\rangle$ stands for the inner product of vectors $\q$ and $\y$.

\begin{remark}
In \eqref{eq: omega}, $\|\cdot\|$ can be any norm on $\R^n$. Same in \eqref{eq: sigma}.
\end{remark}

It can be deduced from Dirichlet's Theorem \cite{Cassels} that
\begin{equation}\label{eq: minkowski}
\ou\geq n \quad \forall \y \in \R^n.
\end{equation}
We call $\y$  very well approximable (abbreviated as VWA) if  $\ou>n$. It is known that the set of VWA vectors has zero Lebesgue measure.
Following \cite{exponent}  the Diophantine exponent $\omega(\mu)$ of a Borel measure $\mu$ is set to be the $\mu$-essential supremum of the $\omega$ function, that is,
\begin{equation}
\label{eq: deu1} \omega(\mu)=\sup \left\{v\mid
\mu\{\y\mid \omega(\y)>v\}>0 \right\}.
\end{equation}
Let  $\MM$ be a smooth submanifold of $\R^n$ and $\mu$ be the measure class of the Riemannian volume on $\MM$. More precisely put, let $\mu$ be
the pushforward $\mathbf{f}_*\lambda$ of $\lambda$(the Lebesgue measure) by any smooth map $\mathbf{f}$ parameterizing $\MM$.  Then the Diophantine exponent of $\MM$, which we denote by $\omega(\MM)$, is set to be equal to $\omega(\mu)$.
$\MM$ is called extremal if  $\omega(\MM)=n$, that is, almost all points of $\MM$ are  not VWA.
A trivial example of an extremal
submanifold of $\R^n$ is $\R^n$ itself.

 K. Mahler \cite{M} conjectured in 1932 that
\begin{equation}\label{eq: curve}
 \MM=\left\{\left.(x,x^2,\ldots,x^n)\right|x \in \R\right\}
\end{equation}
 is an extremal submanifold. This was proved by
Sprind\u{z}uk \cite{Sp1} in 1964. The curve  \eqref{eq: curve} has a notable property that it does not lie in
 any proper affine subspace of $\R^n$. We might describe and formalize this property in terms of nondegeneracy condition as follows.
 Let $\mathbf{f}=(f_1,\ldots,f_n): U\rightarrow \R^n$ be a differentiable map where $U$
is an open subset of $\R^d$. $\mathbf{f}$ is called nondegenerate in
an affine subspace $\LL$ of $\R^n$ at $\x \in U$ if $\mathbf{f}(U)
\subset \LL$ and the span of all the partial derivatives of
$\mathbf{f}$ at $\x$ up to some order coincides with the linear part
of $\LL$. If $\MM$ is a $d$ dimensional submanifold  of $\LL$ we will say
that $\MM$ is nondegenerate in $\LL$ at $\y \in \MM$ if some diffeomorphism
of $\mathbf{f}$ between an open subset $U$ of $\R^d$ and a
neighborhood of $\y$ in $\MM$ is nondegenerate in $\LL$ at
$\mathbf{f}^{-1}(\y)$. We will say $\MM$ is nondegenerate in $\LL$ if it
is nondegenerate in $\LL$ at almost all points of $\MM$.

   It was conjectured by Sprind\u{z}uk \cite{Sp2}
in 1980 that almost all points on a nondegenerate analytic submanifold of $\R^n$ are not very well approximable.
In 1998 D. Kleinbock and  G.A. Margulis   proved in \cite{KM}

\begin{theorem}
Let $\MM$ be a smooth nondegenerate submanifold of $\R^n$, then $\MM$ is extremal.
\end{theorem}

\cite{extremal} studied the conditions under which an affine subspace $\LL$ of $\R^n$ is  extremal and
showed that $\LL$ is extremal if and only if its nondegenerate submanifolds are extremal.
\cite{exponent} derived formulas for  computing $\omega(\LL)$ and $\omega(\MM)$ when $\LL$ is not extremal and $\MM$ is an arbitrary nondegenerate
submanifold in it. This breakthrough was achieved through sharpening of some nondivergence estimates in the space of unimodular lattices (see Lemma
 \ref{lem:  thm2.2}  for review). We record \cite[Theorem 0.3]{exponent} as follows:

\begin{theorem}\label{thm: exponent}
If $\LL$ is an affine subspace of  $\R^n$ and $\MM$ is a nondegenerate
submanifold in $\LL$, then
\begin{equation}
\omega(\MM)=\omega(\LL)=\inf  \left\{\omega(\x)\mid\x \in \LL\right\}=\inf \left\{\omega(\x)\mid\x \in \MM\right\}.
\end{equation}
\end{theorem}

In this paper we will be dealing with multiplicative version of the above concepts. We define
\begin{equation} \label{distance}
 \Pi_+(\y)\stackrel{\mathbf{def}}{=}\prod_{i=1}^{n}|y_i|_+, \quad \text{where}\quad |y_i|_+=\max\left(1,|y_i|\right),
\end{equation}
\begin{equation}
\label{omegatimes} \omega^\times(\y)=\sup \left\{v\left|\text{ }
\exists \infty \text{ many } \q \in \Z^n \text{ with } |
\langle\q,\y\rangle+p|<{\Pi_+(\q)}^{-v/n} \right.\text{ for some } p \in \Z\right\}.
\end{equation}
In the spirit of \eqref{eq: deu1} we define multiplicative Diophantine exponents of manifolds and measures as
 \begin{equation}
 \label{eq: deusigma}
 \omega^\times(\MM)=\omega^\times(\mu)\stackrel{\mathbf{def}}{=}\sup \left\{v\mid \mu\{\y\mid\omega^\times(\y)>v\right\}>0 \},
 \end{equation}
where $\mu$ is the measure class of Riemannian volume on $\MM$.

From definitions we derive $\omega^\times(\y)\geq \omega(\y)$. We call $\y$  very well multiplicatively
approximable (VWMA) if $\omega^\times(\y)>n$. It can be proved that the set of VWMA vectors has zero Lebesgue
measure. Following the terminology of \cite{Sp2}, we call $\MM$ strongly extremal if almost all $\y \in \MM$ are not VWMA.
Strong extremality implies extremality, and to prove a manifold to be strongly extremal is often more difficult to
prove it to be just extremal.

 A. Baker conjectured that the curve \eqref{eq: curve} is strongly extremal \cite{B}
in 1975. Proof of this conjecture was based on dynamical approach proposed in \cite{KM}.
\cite{KM} also proved that nondegenerate manifolds of $\R^n$ are strongly extremal.
In \cite{extremal}  D. Kleinbock gave a necessary and sufficient condition for an arbitrary affine subspace
 to be strongly extremal and  showed that strong extremality of an affine space is  inherited by its nondegenerate submanifolds.
\cite{extremal} also showed that a subspace is strongly extremal iff it contains at least one not VWMA vector.
\cite{survey} gave a detailed account of historical and recent development in the study of multiplicative Diophantine
approximation, and in particular the renowned Littlewood's conjecture \cite[\S 5]{survey}.

This paper will calculate multiplicative
Diophantine exponents of  hyperplanes and their nondegenerate
submanifolds.  We follow  the strategy of associating
 Diophantine property of vectors with behavior of certain trajectories in the space of lattices \cite{KM, extremal}. In this process we will be
 considering multiparameter actions as opposed to one parameter ones which work well for standard Diophantine approximation problems.
Combined with dynamics we use nondivergence estimates in its strengthened format \cite{exponent} (see  Lemma \ref{lem:  thm2.2} of \S3)
to  prove the following:

\begin{theorem}\label{thm: mainthm}
If $\LL$ is a hyperplane of  $\R^n$ and $\MM$ is a nondegenerate
submanifold in $\LL$, then
\begin{equation}
\omega^\times(\LL)=\omega^\times(\MM)=\inf  \left\{\omega^\times(\x)\mid\x \in \LL\right\}=\inf \left\{\omega^\times(\x)\mid\x \in \MM\right\}.
\end{equation}
\end{theorem}

Theorem \ref{thm: mainthm} shows that  multiplicative  Diophantine exponents of hyperplanes are inherited by their nondegenerate submanifolds. We will also
calculate explicitly Diophantine exponents of such spaces in terms of the coefficients of their parameterizing maps.
 In \S4 we will establish

\begin{theorem}\label{thm: oml}
Let $\LL$ be a hyperplane of  $\R^n$ defined by
\begin{equation}\label{eq: ldefi}
(x_1,x_2,\ldots,x_{n-1})\rightarrow
\left(a_1x_1+a_2x_2+\ldots+a_{n-1}x_{n-1}+a_n,x_1,x_2,\ldots,x_{n-1}\right).
\end{equation}
Denote vector $(a_1,\ldots,a_{n-1},a_n)\in \R^n$ by $\mathbf{a}$.
Suppose that $s-1$ is equal to the number of
nonzero elements in $\{a_1,\ldots,a_{n-1}\}$. Then we have
\begin{equation}\label{result1}
\omega^\times(\LL)=\max \left(n,\dfrac{n}{s}\sigma(\mathbf{a}) \right),
\end{equation}
where
\begin{equation}
\label{eq: sigma} \sigma(\mathbf{a})=\sup \left\{v\mid\exists \;\infty
\text{ many } q \in \Z \text{ with } \|q\mathbf{a}+\p\|<|q|^{-v}
\text{ for some } \p \in \Z^n\right\}.
\end{equation}
\end{theorem}

From Theorem \ref{thm: oml} we see that multiplicative Diophantine exponents of $\LL$ and its nondegenerate submanifolds
are dependent on
the parameter $s$. Moreover $s$ takes on integral values from 1 to $n$ and is dependent on the first $n-1$ terms of $\mathbf{a}$ while unaffected by the last term $a_n$.

By comparison as a special case of   \cite[Theorem 0.2]{exponent}, for hyperplane $\LL$ described by \eqref{eq: ldefi}, we have
\begin{equation}\label{result2}
\omega(\LL)=\max\left(n,\sigma(\mathbf{a})\right).
\end{equation}
Consequently
\begin{equation}
\omega^\times(\LL)=\omega(\LL) \quad \text{iff}\quad s=n \quad \text{iff} \quad a_1a_2\cdots a_{n-1}\neq 0;
\end{equation}
\begin{equation}
\omega^\times(\LL)>\omega(\LL) \quad \text{iff} \quad s<n \quad \text{iff} \quad a_1a_2\cdots a_{n-1}= 0.
\end{equation}
In this way we exhibit classes of affine subspaces which are extremal but not strongly extremal.
The main result of this paper is actually much more general than Theorem \ref{thm: mainthm}. We will be considering maps from Besicovitch metric
spaces endowed with Federer measures (we postpone definitions of terminology till \S3).

\section{Dynamics}

We will study homogeneous dynamics and how it relates to Diophantine approximation of vectors. First we define the space of unimodular lattices as follows:
\begin{equation}
\Omega_{n+1}\stackrel{\mathbf{def}}{=}\SL(n+1,\R)\diagup \SL(n+1,\Z).
\end{equation}
 $\Omega_{n+1}$ is noncompact,   and can be decomposed as
\begin{equation} \label{eq: space}
\Omega_{n+1}=\bigcup_{\epsilon>0}K_\epsilon,
\end{equation}
where
\begin{equation}
K_\epsilon=\Big\{\Lambda \in \Omega_{n+1}\left|\; \left\|v\right\| \geq
\epsilon \text{  for all nonzero   }v \in \Lambda\right.\Big\}.
\end{equation}
Each $K_\epsilon$ is compact by Mahler's compactness criterion.

\begin{remark}
$\|\cdot\|$ can be any norm on $\R^{n+1}$ and any two such norms are equivalent. We assume that it is the Euclidean norm from now on.
\end{remark}

We set
\begin{equation}\label{fgt}
g_{\mathbf{t}}=\diag\Big\{e^{-t_1},\ldots,e^{-t_n},e^{t}\Big\}\in \SL(n+1,\R),
\end{equation}
where
\begin{equation}\label{ew}
t_i\geq 0,\quad t=\sum_{i=1}^n t_i,\quad \mathbf{t}=(t_1,\ldots, t_n).
\end{equation}
Also set
\begin{equation}\label{eq: uofy} u_{\y}=\left(\begin{array}{cc} I_n & 0\\
\y& 1\end{array}\right).
\end{equation}
The lattice $u_{\y}\Z^{n+1}$ takes on the form
\begin{equation}\label{eq: uyzn}
u_{\y}\Z^{n+1}=\left\{\left. \left(
          \begin{array}{c}
            \q\\
            \q\y+p
          \end{array}
        \right)\right|\;\q \in \Z^{n}, \p \in \Z
\right\}.
\end{equation}
Also we define
\begin{equation}\label{otv}
W^\times_v \stackrel{\mathrm{def}}{=}\Big\{\y \in \R^n \left|\;\omega^\times(\y)\geq v\right.\Big\}.
\end{equation}
By definition
\begin{equation}
\ot=\sup \Big\{v\left|\;\y \in W^\times_v\right.\Big\}.
\end{equation}
When we have $g_{\mathbf{t}}$ act on  vectors in $u_{\y}\Z^{n+1}$ as defined by \eqref{eq: uyzn}, the first $n$ components will be contracted and the last one  expanded. We propose the following lemma which shows a  correlation between $\omega^\times(\y)$
and trajectories of certain lattices in $\Omega_{n+1}$. The original format stems from \cite[Lemma 5.1]{extremal}, but what we need here is stronger and more precise.

\begin{lem}\label{lem: equiva1}
Suppose we are given a positive integer $k$ $(1\leq k \leq n)$ and a subset $E$ of $\R\times \Z^{n+1}$
which is discrete and homogeneous with respect to
positive integers, and  satisfies the condition that for every $(x,\z)\in E$,  exactly $k$ entries
of $\z$ are nonzero. Take $v>n$ and $c_k=\frac{v-n}{kv+n}$, then
the following are equivalent:
\begin{enumerate}
\item[(i)]
 $\exists (x,\z) \in E$  with arbitrarily large $\|\z\|$ such that
\begin{equation}\label{lowerotimes}
|x| \leq {\Pi_+(\z)}^{-v/n}
\end{equation}
\item[(ii)] $\exists$ an unbounded
set of $\mathbf{t}\in \R_+^n$ such that for some  $(x,\z) \in E
\backslash\{0\}$ we have
\begin{equation}\label{omegaequal}
 \max\Big(e^{t}|x|,\quad e^{-t_i}|z_i|\Big) \leq
e^{-c_kt},\quad 1\leq i\leq n
\end{equation}
\end{enumerate}
\end{lem}

\begin{proof}
Suppose (i) holds. Without loss of
generality, assume $|z_i|\geq 1$ for $i\leq  k$ and $z_i=0$
for $i>k$. Define $t$ by
\begin{equation} \label{first}
e^{(1-kc_k)t}=\Pi_+(\z)=|z_1\ldots z_k|.
\end{equation}
Note that $c_k<1/k$ from its definition $c_k=\frac{v-n}{kv+n}$, and $t$ defined in the above equation is nonnegative thereof.

Then for every $t$ define $t_i$ by
\begin{equation}\label{second}
e^{-t_i}|z_i|=e^{-c_kt}\quad \text{if} \quad 1\leq i\leq k,\quad t_i=0
\quad \text{if}\quad i>k.
\end{equation}
Note that from \eqref{first} and \eqref{second} it is verified that
$t=\sum_{i=1}^kt_i$.  And we have
\begin{equation}
e^{t}|x|\leq
e^t{\Pi_+(\z)}^{-v/n}=e^te^{(1-kc_k)(-v/n)t}=e^{t+(1-kc_k)(-v/n)t}.
\end{equation}
Plugging in $c_k=\frac{v-n}{kv+n}$, we get
\begin{equation}
1+ (1-kc_k)(-v/n)=-c_k.
\end{equation}
Hence
\begin{equation}
e^{t+(1-kc_k)(-v/n)t}=e^{-c_kt}.
\end{equation}
Hence (ii) is satisfied. In addition,  by taking $\|\z\|$ arbitrarily large we
produce arbitrarily large $\Pi_+(\z)$ and $t$ from \eqref{first}.

Suppose (ii) holds. Because $(x,\z) \in E$ by reordering entries
of $\z$ such that $|z_i|\geq 1$ \;for\; $i\leq k$ and
$z_i=0$ for $i>k$, we have
\begin{equation}
|z_i|\leq e^{t_i-c_kt} \quad \text{if}\quad  i\leq k,\qquad |x|\leq
e^{-(1+c_k)t}.
\end{equation}
\begin{equation}
\Pi_+(\z)=|z_1\ldots z_k|\leq e^{(t_1-c_kt)+(t_2-c_kt)+\ldots+(t_k-c_kt)}=
e^{t_1+\ldots+t_k-kc_kt}\leq e^{t-kc_kt}.
\end{equation}
By plugging in $c_k=\frac{v-n}{kv+n}$, we get
\begin{equation}
e^{-(1+c_k)t}=(e^{(1-kc_k)t})^{-v/n}.
\end{equation}
 Hence
\begin{equation}
|x|\leq
e^{-(1+c_k)t}=(e^{(1-kc_k)t})^{-v/n}\leq {\Pi_+(\z)}^{-v/n}.
\end{equation}
Also by the discreteness of $E$, if $\|\z\|$ has a uniform bound while $|x|$ tends to zero, $(0,\z_0) \in E$ for some nonzero $\z_0$ and any integral multiple of
$(0,\z_0)$ will satisfy \eqref{lowerotimes}. Obviously  $\|p\z_0\|$ tends to infinity when the integer $p$ tends to infinity.  Therefore (i) is
established.
\end{proof}

\begin{remark}\label{remark1}
In \eqref{omegaequal},
because $|z_i|\leq  e^{t_i-c_kt} $, we have  $t_i-c_kt\geq 0$ for at least $k$ values of $i$. This information is
important because of the following elementary observation which plays an indispensable role in the proof of Lemma \ref{lem: negation2} in
\S 4:
\end{remark}

\begin{lem}\label{lem: elementary}
Suppose  $p\in \Z$ and $|p|\leq  e^\alpha$.
If $\alpha \geq 0$ then we have $|p|_+\leq  e^\alpha$.
\end{lem}

\begin{proof}
From \eqref{distance} directly.
\end{proof}

\begin{remark}
If $\alpha<0$, then $|p|\leq  e^\alpha$ does not imply $|p|_+\leq  e^\alpha$. This distinction is important because in multiplicative Diophantine approximation we think of $|p|_+$ instead of $|p|$.
\end{remark}

We define
\begin{equation}\label{eq: decompose}
 \Z^{n+1}_{k}=\left\{\left.(\q,p)=(q_1,\ldots,q_n,p)\in
\Z^{n+1}\right|\text{ exactly $k$ entries of }\q \text{ are nonzero}\right\}.
\end{equation}
Apparently
\begin{equation}
 \Z^{n+1}=\bigcup_{k=0}^n \Z^{n+1}_{k}.
\end{equation}
In light of  Lemma \ref{lem: equiva1}, if we set $v>n$, $\y \in \R^n$ and
 \\$E=\left\{(|\langle\q,\y\rangle+p|,\q )\left|(\q,p)\in \Z^{n+1}_k\right.\right\}$,
 condition (i) of Lemma \ref{lem: equiva1} implies that
\begin{equation}\label{add1}
 \y \in W^\times_{v}.
\end{equation}
Condition (ii) becomes equivalent to:  $\exists$  an unbounded set of  $\mathbf{t} \in \R_+^{n}$ such that
\begin{equation}\label{eq: restriction}
t_i \geq c_kt \text{   for at least $k$ values of  } i,
\end{equation}
and
\begin{equation} \label{eq: omegaequa2}
g_\mathbf{t}u_\y\Z^{n+1}_k \text{ contains at least one  vector with norm } \leq e^{-c_kt}.
\end{equation}
Furthermore
\begin{equation}\label{eq: ck}
c_k=\dfrac{v-n}{kv+n}\Longleftrightarrow  v=\dfrac{n+nc_k}{1-kc_k},\qquad 1\leq k \leq n
\end{equation}
Recall that by \eqref{ew} $\mathbf{t}$ is multiparameter vector in $\R_+^{n}$ and $t=\sum_{i=1}^n t_i$.
If we set
\begin{equation}\label{gammak}
\gamma_k(\y)= \sup \Big\{c_k \mid\eqref{eq: omegaequa2} \text{ holds  for an unbounded set of } \mathbf{t} \in \R_+^{n} \text{ satisfying }\eqref{eq: restriction}  \Big\},
\end{equation}
we have the following theorem, which is the main result of this section.

\begin{theorem}
 $\forall \y \in \R^n$,  we have
\begin{equation} \label{eq: otyequal}
\omega^\times(\y)=\max_{\substack{1\leq k \leq n}}\dfrac{n+n\gamma_k(\y)}{1-k\gamma_k(\y)}.
\end{equation}
\end{theorem}

\begin{proof}
We first prove that in \eqref{gammak} we can have
$\mathbf{t} \in \Z_{+}^n$ as opposed to $\mathbf{t} \in \R_+^{n}$. We adopt arguments of
\cite[Corollary 2.2]{KM} here.
Suppose for some $\mathbf{t}=(t_1,t_2,\ldots,t_n)$ and $t=\sum t_i$ , we have
$t_i \geq c_kt \text{   for at least $k$ values of  } i$ as well as
\[g_\mathbf{t}u_\y\Z^{n+1}_k \text{ contains at least one  vector with norm } \leq e^{-c_kt}.\]
Denote by $[\mathbf{t}]$ the vector consisting of integer parts of $t_i$. Then the ratio of lengths of the shortest vector of
$g_{[\mathbf{t}]}u_\y\Z^{n+1}_k$ and the shortest vector of $g_{\mathbf{t}}u_\y\Z^{n+1}_k$ is bounded from above by
\[\left\|g_{\mathbf{t}}g_{[\mathbf{t}]}^{-1}\right\|=\left\|g_{\mathbf{t}-[\mathbf{t}]}\right\|\leq e^n.\]
Hence we get
\[g_{[\mathbf{t}]}u_\y\Z^{n+1}_k \text{ contains at least one  vector with norm } \leq e^n e^{-c_kt}.\]
When $t$ is large we can decrease $c_k$ slightly to $c_k'$ and get
\[g_{[\mathbf{t}]}u_\y\Z^{n+1}_k \text{ contains at least one  vector with norm } \leq e^{-c_k't}\]
as well as
\[[t_i] \geq c_k'\left(\sum[t_i]\right) \text{   for at least $k$ values of  } i.\]
Therefore
\[\gamma_k(\y)= \sup \Big\{c_k \mid\eqref{eq: omegaequa2} \text{ holds  for an unbounded set of } \mathbf{t} \in \Z_{+}^n \text{ satisfying }\eqref{eq: restriction} \Big\}.\]
Next we show
\[ \omega^\times(\y)\geq \dfrac{n+n\gamma_k(\y)}{1-k\gamma_k(\y)},  \quad 1\leq k \leq n.\]
To see this,  apply Lemma \ref{lem: equiva1} $n$ times, letting $k$ go from $1$ to $n$. For each $k$, condition (ii)  of Lemma \ref{lem: equiva1} implies
 condition (i), which in turn implies  that
$\y \in W^\times_{v}$  or $\omega^\times(\y)\geq v$.

On the other hand, \eqref{add1} clearly forces condition (i) of Lemma \ref{lem: equiva1} to
hold for some $k$ between $1$ and $n$. Hence
\[ \omega^\times(\y)\leq\max_{\substack{1\leq k \leq n}}\dfrac{n+n\gamma_k(\y)}{1-k\gamma_k(\y)}.\]
\eqref{eq: otyequal} is therefore established.
\end{proof}

Suppose $\nu$ is a measure on $\R^n$ and $v> n$ , by
definition
\begin{equation}
\omega^\times(\nu)\leq v
\quad \text{if and only if} \quad \nu(W^\times_u)=0, \quad \forall u>v.
\end{equation}
By the Borel-Cantelli Lemma and the above theorem, a sufficient condition for
$\omega^\times(\nu)\leq v$ is:

\begin{condition}\label{condition: upperboundot}
$\forall k$ $(1\leq k\leq n)$, $\forall d_k>c_k$,  we have
\begin{equation}\label{eq: upperboundot}
\sum_{\substack{\mathbf{t} \in \Z_{+}^n,\\ t_i\geq d_kt \text{ for
at least }\\ k \text{ values of }i}}\nu\Big(\Big\{\y\left|g_{\mathbf{t}}u_\y\Z^{n+1}_k \text{ has at least one nonzero vector with norm} \leq e^{-d_kt}\right.
\Big\}\Big)<\infty.
\end{equation}
\end{condition}

\begin{remark}
Condition \ref{condition: upperboundot} is helpful because it allows us to find upperbounds of $\omega^\times(\lambda)$ by applying  quantitative nondivergence in the next section. The restriction similar to \eqref{eq: restriction} will be used in the proof of Lemma \ref{lem: negation2} in \S4.
\end{remark}

\section{Quantitative Nondivergence}

Before stating nondivergence quantitative results, we
first introduce an assembly of relevant concepts developed in \cite{exponent}, \cite{KLW} and \cite{KM}.
A metric space $X$ is called $N-Besicovitch$ if for any bounded
subset $A$ and any family $\beta$ of nonempty open balls of $X$ such
that each $x \in A$ is a center of some ball of $\beta$, there is a
finite or countable subfamily $\{\beta_i\}$ of $\beta$ covering $A$
with multiplicity at most $N$. $X$ is $Besicovitch$ if it is $N-Besicovitch$ for some $N$.

Let $\mu$ be a locally finite Borel measure on $X$, $U$ an open
subset of $X$ with $\mu(U)>0$. Following \cite{KLW} we call $\mu$  $D-Federer$ on $U$ if
\begin{equation}
\sup _{\begin{subarray}{1}x \in \textrm{supp }\mu,\; r>0
\\ B(x,3r)\subset U  \end{subarray}}\dfrac{\mu(B(x,3r))}{\mu(B(x,r)}<D
\end{equation}
$\mu$ is said to be $Federer$ if for $\mu$-a.e. $x \in X$ there
exists a neighborhood $U$ of $x$ and $D>0$ such that $\mu$ is
$D-Federer$ on $U$.

An important illustration of the above notions is that $\R^d$ is Besicovitch and $\lambda$, the Lebesgue measure
is $Federer$. Many natural measures supported on fractals are also known to be $Federer$ (see \cite{KLW} for technical details).

For a subset $B$ of $X$ and a function $f$ from $B$ to a normed space with norm $\|\textrm{ } \|$, we define
  $\|f\|_B= \sup_{x \in B}\|f(x)\|$. If $\mu$ is a
  Borel measure on $X$ and $B$ a subset of $X$ with $\mu(B)>0$
  $\|f\|_{\mu, B}$ is set to be $\|f\|_{B \cap \textrm{supp }\;\mu}$.

 A function $f: X\rightarrow \R$ is called $(C,\alpha)$-good on $U \subset X$
with respect to $\mu$ if for any open ball $B$ centered in supp
$\mu$ one has
\begin{equation}
\forall \varepsilon>0 \qquad  \mu(\{x \in B\mid |f(x)|<\varepsilon\})\leq
C\left(\dfrac{\varepsilon}{\|f\|_{\mu,B}}\right)^{\alpha}\mu(B).
\end{equation}
Roughly speaking a function is $(C,\alpha)$-good if the set of points where it takes small value
has small measure. In Lemma \ref{lem: thm2.2} we use the fact that functions of the form $\x\rightarrow
\|h(\x)\Gamma\|$ , where $\Gamma$ runs through subgroups of $\Z^{n+1}$, are $(C, \alpha)$-good with uniform $C$ and $\alpha$.

Let $\mathbf{f}=(f_1,\ldots,f_n)$ be a map from $X$ to $\R^n$. Following \cite{exponent} we say that ($\mathbf{f}, \mu$) is good
at $x \in X$ if there exists a neighborhood $V$ of $x$ such that any linear combination of $1,f_1,\ldots,f_n$
is $(C,\alpha)$-good on $V$ with respect to $\mu$ and ($\mathbf{f}, \mu$) is good if ($\mathbf{f}, \mu$) is good at $\mu$-almost
every point. Reference to measure will be omitted if $\mu=\lambda$,  and we will simply say that $\mathbf{f}$
is good or good at $x$. For example polynomial maps are good. \cite{extremal} proved the following result:

\begin{lem}\label{lem: goodmap}
Let $\LL$ be an affine subspace of $\R^n$, and let $\mathbf{f}$ be a smooth map from $U$, an open subset of $\R^d$ to $\LL$
which is nondegenerate at $\x \in U$; then $\mathbf{f}$ is good at $\x$.
\end{lem}

Furthermore if $\LL$ is an affine subspace of $\R^n$ and $\mathbf{f}$
a map from X into $\LL$, following \cite{exponent} we say ($\mathbf{f}, \mu$) is nonplanar in $\LL$ at $x
\in \textrm{supp }\mu$ if $\LL$ is equal to the intersection of all
affine subspaces containing $\mathbf{f}(B\cap \textrm{supp }\mu)$ for any open neighborhood $B$ of $x$.
($\mathbf{f}, \mu$) is nonplanar in $\LL$ if ($\mathbf{f}, \mu$) is nonplanar in $\LL$ at $\mu$-a.e. $x$. We skip saying $\mu$ when $\mu=\lambda$ and
skip $\LL$ if $\LL=\R^n$. From definition
($\mathbf{f}, \mu$) is nonplanar if and only if for any open $B$ of positive measure, the restrictions of $1,f_1,\ldots,f_n$ to $B\cap \textrm{supp }\mu$ are linearly independent over $\R$. Clearly nondegeneracy in $\LL$ implies nonplanarity in $\LL$.  Nondegenerate smooth maps from $\R^d$ to
$\R^n$ as in Lemma \ref{lem: goodmap} give typical examples of nonplanarity.

Let $\Gamma$ be any discrete subgroup of $\R^{k}$  we denote by $rk(\Gamma)$ the rank of $\Gamma$ when viewed as a $\Z$-module.
The following is exactly  \cite[ Theorem 2.2]{exponent}.

\begin{lem}\label{lem:  thm2.2}
Let $m$, $N \in \N$ and $C,D,\alpha,\rho >0$ and suppose we are
given an $N-Besicovitch$ metric space $X$, a ball $B=B(x_0,
r_0)\subset X$, a measure $\mu$ which is $D-Federer$ on
$\tilde{B}=B(x_0, 3^mr_0)$ and a map $h\colon \tilde{B}\to
\GL_m(\R)$. Assume the following two conditions hold:
\begin{enumerate}
\item[(i)] $\forall \;\Gamma \subset \Z^m$,  the function $ x\rightarrow
\left\|h(x)\Gamma\right\|$ is $(C, \alpha)$-good on $\tilde{B}$ with respect to
$\mu$;

\item[(ii)] $\forall \; \Gamma \subset \Z^m$, $\left\|h(\cdot )\Gamma\right\|_{\mu, B}\geq\rho^{rk(\Gamma)}$.
\end{enumerate}
Then for any positive $\epsilon \leq \rho$, we have
\begin{equation}
\mu\Big(\Big\{x \in B\mid h(x)\Z^m \notin K_{\epsilon}\Big\}\Big)\leq
mC(ND^2)^m\left(\dfrac{\epsilon}{\rho}\right)^{\alpha}\mu(B).
\end{equation}
\end{lem}

\begin{prop}\label{prop: proposition}
Let $X$ be a Besicovitch metric space, $B=B(\x, r)\subset X$, $\mu$ a
measure which is $D-Federer$ on $\tilde{B}=B(\x, 3^{n+1}r)$ for some
$D>0$ and $ \mathbf{f}$ a continuous map from $\tilde{B}$ to $\R^n$.
Given $v\geq n$, let $c_k=\frac{v-n}{kv+n}$ where $1\leq k \leq n$ and assume that
\begin{enumerate}
\item[(i)]
$\exists C,\alpha >0$ such that all the functions   $\x\to
\left\|g_{\mathbf{t}}u_{\mathbf{f}(\x)}\Gamma\right\|$, $\Gamma \subset \Z^{n+1}$ are $(C,
\alpha)$- good on  $\tilde{B}$ with respect to $\mu$
\item[(ii)]
  $\forall k$ $(1\leq k \leq n), \quad \forall \;  d_k>c_k$, $\exists T=T(d_k)>0$ such that for any vector $\mathbf{t} \in \Z_{+}^{n}$ with  $t\geq T$
  and $t_i\geq d_kt$    for at least $k$ values of  $i$ and
 any $\Gamma \subset \Z^{n+1}$,  we have
 \begin{equation}\label{eq: critical}
\left\|g_{\mathbf{t}}u_{\mathbf{f}(\cdot)}\Gamma\right\|_{\mu,
 B}\geq e^{-rk(\Gamma)d_kt}.
 \end{equation}
\end{enumerate}
Then $\omega^\times(\mathbf{f}_*(\mu|_B))\leq v$.
\end{prop}

\begin{proof}
Apply Lemma \ref{lem: thm2.2} $n$ times, letting $k$ go from $1$ to $n$. For each iteration set  $m=n+1$
and $\nu=\mathbf{f}_*(\mu|_B)$.

$\forall k$, $\forall d_k>c_k $ and for all
 $\mathbf{t} \in \Z_{+}^n$  satisfying the condition that $ t_i\geq d_kt$ for at least $k$ values of $i$,
set $h_k(\x)=g_{\mathbf{t}}u_{\mathbf{f}(\x)}$.
We see that condition (i) of Lemma \ref{lem: thm2.2}
agrees with condition (i) of Proposition \ref{prop: proposition}.

For the other condition,  set
$\rho_k^t=e^{-\frac{c_k+d_k}{2}t}$ and $\epsilon_k^t=e^{-d_kt}$. Note that
\[d_k > c_k \Leftrightarrow \epsilon_k^t<\rho_k^t.\]
 Also we have
\[\frac{\epsilon_k^t}{\rho_k^t}=e^{-\frac{d_k-c_k}{2}t}.\]
It follows that
condition (ii) of Proposition \ref{prop: proposition} implies condition (ii) of Lemma \ref{lem: thm2.2} for $t>T(\frac{c_k+d_k}{2})$.
Hence by Lemma \ref{lem: thm2.2}, for any fixed $\mathbf{t} \in \Z_{+}^n$ with $t\geq T$ and $ t_i\geq d_kt$ for at least $k$ values of $i$, we  have
\begin{align}\label{id2}
\nu\left(\left\{\y\mid g_{\mathbf{t}}u_{\y}\Z^ {n+1} \notin
K_{e^{-d_kt}}\right\}\right)&=\mu\left(\left\{\x \in B \mid h_k(\x)\Z^ {n+1} \notin
K_{e^{-d_kt}}\right\}\right)\\
&\leq const \cdot e^{-\alpha\frac{d_k-c_k}{2}t}\mu(B).\nonumber
\end{align}
We have the obvious identity
\begin{equation}\label{id1}
\sum_{\mathbf{t} \in \Z_{+}^n}^{\infty}\nu\Big(\left\{\y \mid g_{\mathbf{t}}u_{\y}\Z^ {n+1} \notin
K_{e^{-d_kt}}\right\}\Big)=
\sum_{l=1}^{\infty}\sum_{\mathbf{t} \in \Z_{+}^n
,\; t=l}\nu\Big(\left\{\y \mid g_{\mathbf{t}}u_{\y}\Z^ {n+1} \notin
K_{e^{-d_kl}}\right\}\Big).
\end{equation}
Since for each $l \in \N$, the possible number of $\mathbf{t} \in \Z_{+}^n$ with $t=l$ is bounded from above by
$(l+1)^n$, we  get from \eqref{id1}
\begin{equation}
\sum_{\substack{\mathbf{t} \in \Z_{+}^n,\\ t_i\geq d_kt \text{ for
at least }\\ k \text{ values of }i}}\nu\Big(\left\{\y \mid g_{\mathbf{t}}u_{\y}\Z^ {n+1} \notin
K_{e^{-d_kt}}\right\}\Big)\leq
\sum_{l=1}^{\infty}(l+1)^n const \cdot e^{-\alpha\frac{d_k-c_k}{2}l}\mu(B)<\infty.
\end{equation}
Since $h_k(\x)\Z^ {n+1}_k \subset h_k(\x)\Z^{n+1}$, we have
\begin{align}
&\left\{\x \in B \mid h_k(\x)\Z^{n+1}_k \text{ has at least one  vector with norm} \leq e^{-d_kt}\right\}
 \nonumber \\
& \subset \left\{\x \in B \mid h_k(\x)\Z^ {n+1} \notin
K_{e^{-d_kt}}\right\}.
\end{align}
Moreover we note that the restriction $t_i\geq d_kt$    for at least $k$ values of  $i$ is also present
in Condition \ref{condition: upperboundot}.
We let $k$ range over all integers between 1  and $n$ and Condition \ref{condition: upperboundot} is satisfied.
\end{proof}

\section{Proof of Main Theorems}

To prove the theorems, we first calculate $\|g_{\mathbf{t}}u_{\mathbf{f}(\cdot)}\Gamma\|_{\mu,
 B}$ in \eqref{eq: critical}. The following exterior algebraic computation
comes from \cite{exponent} and \cite{KM}.

Suppose $\R^{n+1}$ has standard basis $
\mathbf{e}_1,\dots,\mathbf{e}_{n+1}$, and if we extend the Euclidean
structure of $\R^{n+1}$ to $\textstyle\bigwedge^j(\R^{n+1})$, then for index sets
\begin{equation}\label{eq: index1}
I=\big\{i_1, i_2,\dots,
i_j\big\}\subset\big\{1,2,\dots, {n+1}\big\}, \quad i_1<i_2<\dots<i_j
\end{equation}
$\left\{\mathbf{e}_I\mid \mathbf{e}_I=\mathbf{e}_{i_1}\wedge
\mathbf{e}_{i_2} \wedge \dots \wedge \mathbf{e}_{i_j}, \quad
\#I=j\right\}$ form an orthogonal basis of
$\bigwedge^j(\R^{n+1})$ when $I$ range over all index sets of the form \eqref{eq: index1}.
If  a discrete subgroup $\Gamma \subset \R^{n+1}$ of rank $j$ is
viewed as a $\Z$-module with basis
$\mathbf{v}_1,\ldots,\mathbf{v}_j$, then we may represent it by  exterior product
$\w=\mathbf{v}_1\wedge\ldots\wedge \mathbf{v}_j$. Observing
$\|\Gamma\|=\|\w\|$,  we will be able to compute
$\left\|g_tu_\mathbf{f}\Gamma\right\|_{\mu,B}$ as in  \eqref{eq: critical}  directly.

We assume from now on that $J$ and $I$
stand for index sets:
$J$ is of order $j-1$ and $I$ is of order $j$.
Given $\y=(y_1,\ldots,y_n)$,
we set  $y_{n+1}=1$ and get $u_\y$ as in \eqref{eq: uofy}. We get
   \begin{align}
   u_\y\mathbf{e}_I&=\mathbf{e}_I,\qquad \textrm{if }n+1 \in I;\nonumber\\
   u_\y\mathbf{e}_I&=\mathbf{e}_I\pm\sum_{i \in I} y_i\mathbf{e}_{I\cup\{n+1\}\setminus\{i\}}\qquad\text{otherwise}.
   \end{align}
Hence
   \begin{equation}\label{eq: uyw}
   u_\y\w=\sum_{\substack{I\subset \{1,\ldots,n\}}}\pm\langle \mathbf{e}_I,\w \rangle \mathbf{e}_I+
   \sum_{\substack{J\subset \{1,\ldots,n\}}}\left(\sum_{i=1}^{n+1}\pm\langle \mathbf{e}_i\wedge \mathbf{e}_J,\w \rangle y_i\right)\mathbf{e}_J \wedge \mathbf{e}_{n+1}.
   \end{equation}
Since $g_\mathbf{t}=\diag\left\{e^{-t_1},\ldots,e^{-t_n},e^{t}\right\}$, we have
\begin{align}
   g_\mathbf{t}\mathbf{e}_i&=e^{-t_i}\mathbf{e}_i\quad  (1\leq i\leq n);\\
   g_\mathbf{t}\mathbf{e}_{n+1}&=e^{t}\mathbf{e}_{n+1};
\end{align}
\begin{align}\label{eq: uyw2}
   g_{\mathbf{t}}u_\y\w=&\sum_{\substack{I}\subset \{1,\ldots,n\}}e^{-\sum_{\substack{i \in I}}t_i}\pm\langle \mathbf{e}_I,\w \rangle \mathbf{e}_I\\
   &+\sum_{\substack{J\subset \{1,\ldots,n\}}}e^{t-\sum_{\substack{i \in J}}t_i}\left(\sum_{i=1}^{n+1}\pm \langle \mathbf{e}_i\wedge \mathbf{e}_J,\w \rangle y_i\right) \mathbf{e}_{J} \wedge \mathbf{e}_{n+1}. \nonumber
\end{align}
For  $\mathbf{f}=(f_1,f_2,\ldots,f_n):\tilde{B} \rightarrow \R^n$ in \eqref{eq: critical}, we set $f_{n+1}=1$ and
\[\widetilde{\mathbf{f}}=(f_1,\ldots,f_n,1).\]
Also set
\begin{equation} \label{eq: cwdefi}
\mathbf{c}(\w)_i=\sum_{\substack{J\subset \{1,\ldots,n\} \\ \# J=j-1}}\pm\langle \mathbf{e}_i\wedge \mathbf{e}_J,\w \rangle\mathbf{e}_J\in \textstyle\bigwedge^{j-1}(\R^{n+1}), \quad 1 \leq i \leq n+1,
\end{equation}
\begin{equation}
  \mathbf{c}(\w)=\left(
                   \begin{array}{c}
                    \mathbf{c}(\w)_1\\
                     \mathbf{c}(\w)_2\\
                     \vdots \\
                     \mathbf{c}(\w)_{n+1} \\
                   \end{array}
                 \right).
  \end{equation}
Noting that
$\mathbf{e}_I$ and $\mathbf{e}_{J} \wedge \mathbf{e}_{n+1}$ appearing in \eqref{eq: uyw2} are orthogonal, we have,
 up to some constant  dependent on $n$ only
\begin{align}\label{eq: compare}
\big\|g_{\mathbf{t}}u_{\mathbf{f}(\cdot)}\w\big\|_{\mu,B}&\asymp\max\left(e^{-\sum_{\substack{i \in I}}t_i}\|\langle \mathbf{e}_I,\w \rangle\|,\quad e^{t-\sum_{\substack{i \in J}}t_i}\left\|\sum_{i=1}^{n+1}\pm\langle \mathbf{e}_i\wedge \mathbf{e}_J,\w \rangle f_i\right\|_{\mu,B} \right)\\
&=\max\left(e^{-\sum_{\substack{i \in I}}t_i}\|\langle \mathbf{e}_I,\w \rangle\|,\quad e^{t-\sum_{\substack{i \in J}}t_i}
\left\|\widetilde{\mathbf{f}}(\cdot  )\mathbf{c}(\w)\right\|_{\mu,B} \right), \nonumber
\end{align}
 where the maximum is taken over all index sets $I\subset \{1,\ldots,n\}$ and $J\subset \{1,\ldots,n\}$.

 Following arguments of \cite{exponent}, we see that
 the  value of $\|g_{\mathbf{t}}u_{\mathbf{f}(\cdot)}\w\|_{\mu,B}$ as in \eqref{eq: compare} is affected
 by the linear dependence relations between
the components of $\widetilde{\mathbf{f}}$.  We denote by $\mathcal{F}_{\mu,B}$ the $\R$-linear span of the restrictions of
$f_1,\ldots,f_n, 1$ to $B\cap \text{ supp } \mu$, denote its dimension by $l+1$, and choose functions $g_1,\ldots,g_l: B\cap \text{ supp } \mu
\rightarrow \R$
 such that $g_1,\ldots,g_l,1$ form a basis of $\mathcal{F}_{\mu,B}$. This choice defines a matrix
 \begin{equation}
 R = (r_{i,j})_{\substack{i=1,\ldots,l+1\\j=1,\ldots,n+1}}\in M_{l+1,n+1}
 \end{equation}
formed by coefficients in the expansion of $f_1,\ldots,f_n,1$  as linear combinations of $g_1,\ldots,g_l,1$.
In other words, with the notation $\widetilde{\mathbf{g}}=(g_1,\ldots,g_l,1)$,  we have
\begin{equation}\label{R}
\widetilde{\mathbf{f}}(\x)=\widetilde{\mathbf{g}}(\x)R, \quad \forall \x \in B\cap \text{ supp } \mu.
\end{equation}
Therefore  $\|\widetilde{\mathbf{f}}(\cdot)\mathbf{c}(\w)\|_{\mu,B}$ can be replaced by
$\left\|\widetilde{\mathbf{g}}(\cdot  )R\mathbf{c}(\w)\right\|_{\mu,B}$ and the latter, in
view of linear independence of the components of $\widetilde{\mathbf{g}}$, simply by the norms of vectors $R\mathbf{c}(\w)$ (up to some
 constant uniform  in  $\w$  yet dependent on $\mathbf{f}$, $\widetilde{\mathbf{g}}$, $\mu$ and $B$). The second assumption of Proposition \ref{prop: proposition} can be rewritten as:

\begin{condition}\label{condition: equi0}
$\forall k$  $(1\leq k \leq n)\quad \forall   d_k>c_k$, $\exists T=T(d_k)>0$ such that for all  $\mathbf{t} \in \Z_{+}^{n}$ with  $t\geq T$
  and $t_i\geq d_kt$    for at least $k$ values of  $i$, we have $\forall j$ $(1\leq j \leq n)$,  $\forall \text{ nonzero }\w \in \textstyle\bigwedge^j(\Z^{n+1})$,
\[\max\left(e^{-\sum_{\substack{i \in I}}t_i}\|\langle \mathbf{e}_I,\w \rangle\|,\quad e^{t-\sum_{\substack{i \in J}}t_i}
\left\|R\mathbf{c}(\w)\right\|\right)\geq e^{-jd_kt},\]
where the maximum is taken over all index sets $I\subset \{1,\ldots,n\}$ with order $j$ and $J\subset \{1,\ldots,n\}$ with order $j-1$. Matrix
$R$ is defined via \eqref{R}.
\end{condition}

\begin{remark}
$k$ and $j$ are independent variables: $k$ arises from Lemma \ref{lem: equiva1}  while $j$ is the rank of $\w$. $R$
depends on  the measure $\mu$, the ball $B$, the map $\mathbf{f}$  as well as the choice of   $\widetilde{\mathbf{g}}$.
\end{remark}

According to \cite{exponent}, the only way the ball $B$, the measure $\mu$ and the map $\f$ enter the above
conditions is via the matrix $R$, which depends on  $B$,  $\mu$ and $\f$ and is
not uniquely determined. However another choice of $R$ would  yield a condition
equivalent to Condition \ref{condition: equi0}. Let $\mu$ be a Federer measure on a Besicovitch metric space $X$ and $\LL$ a hyperplane of $\R^n$. We assume from now on that
 $\f\colon X\to \LL$ is a continuous map  such that  $(\f,\mu)$ is nonplanar in $\LL$. For a subset $M$ of $\R^n$,
 define its affine span  $\langle M\rangle_a$ to be the intersection of all affine subspaces of $\R^n$ containing $M$.
 By definition \cite[\S1]{exponent}, $(\f,\mu)$ is   nonplanar in $\LL$ iff
 \begin{equation}
 \LL=\langle\mathbf{f}(B \cap \text{ supp } \mu)\rangle_a, \quad \forall \text{ open } B  \subset X \text{ with } \mu(B)>0.
\end{equation}
 Suppose
\begin{equation}
\mathbf{h}\colon \R^{n-1} \to \LL=\langle\mathbf{f}(B \cap \text{ supp } \mu)\rangle_a \text{ is an affine isomorphism, and}\nonumber
\end{equation}
\begin{equation}\label{hmap}
\widetilde{\mathbf{h}}(\x)=\widetilde{\mathbf{x}}R, \quad \mathbf{x}\in  \R^{n-1},
\end{equation}
where $\widetilde{\mathbf{h}}=(h_1,\ldots,h_n,1)$ and $\widetilde{\mathbf{x}}=(x_1,\ldots,x_{n-1},1)$. Then $R$ and
$\mathbf{g}=\mathbf{h}^{-1}\circ \mathbf{f}$ satisfy \eqref{R}.  $g_1,\ldots, g_{n-1},1$ generate
$\mathcal{F}_{\mu,B}$ and are linearly independent over $\R$. This way Condition \ref{condition: equi0} or the second assumption of Proposition \ref{prop: proposition} becomes a property of the space $\langle\mathbf{f}(B \cap \text{ supp } \mu)\rangle_a$ or  $\LL$.
We can thus choose $R$ uniformly for all  measures $\mu$,  balls $B$ and maps $\mathbf{f}$.
Since the statement that Condition \ref{condition: equi0} holds for any $R$ satisfying \eqref{R} is equivalent to the statement that it
holds for some  $R$ satisfying \eqref{R}, we will make the most natural choices for  $\LL$  as
described in \eqref{eq: ldefi}: $X=\R^{n-1}$,
$\mu=\lambda$ and  the following  map according to  \eqref{hmap}:
\begin{equation} \label{eq: rdefi1}
   \widetilde{\mathbf{h}}(\x)=(h_1,\ldots,h_n,1)(\x)=(x_1,x_2,\ldots,x_{n-1},1)R_0,
   \end{equation}
   where  $R_0$ is an $n\times(n+1)$ matrix defined by
   \begin{equation}\label{eq: rdefi}
R_0=\left(  \begin{array}{c}
a_1\\ \vdots \\  a_{n-1} \\a_n \\ \end{array}  I_n\right).
\end{equation}
We can replace an arbitrary $R$ in Condition \ref{condition: equi0} by $R_0$ defined in \eqref{eq: rdefi} as long as
$(\f,\mu)$ is nonplanar in $\LL$.

Noting \eqref{eq: compare}  and the fact that $e^{t-\sum_{\substack{i \in J}}t_i}\geq 1$,  we get
\begin{equation}\label{eq: critical1}
   \left\|g_tu_{\mathbf{f}(\cdot)} \w\right\|_{\mu,B} \succ \left\|R_0\mathbf{c}(\w)\right\|,
\end{equation}
where $\succ$ implies some constant dependent on $\mu$, $B$ and $\mathbf{f}$.

Next we restate and reprove \cite[Lemma 4.6]{extremal}.

\begin{lem}
For $R_0$ defined in \eqref{eq: rdefi} and nonzero $\w \in \textstyle\bigwedge^j(\Z^{n+1})$, we have
  \begin{equation}\label{eq: simple}
  \left\|R_0\mathbf{c}(\w)\right\|\geq 1 \text{ if } j>1.
  \end{equation}
  \end{lem}

\begin{proof}
Suppose for some index set $I_1=\{i_1,i_2,\ldots,i_j\}$ we have  $a=\langle \mathbf{e}_{I_1},\w\rangle\in \Z$ and $a\neq 0$. Since $j>1$, without loss
of generality, we assume that $i_1=1$ and $i_2=2$.
We consider the first entry of $\left\|R_0\mathbf{c}(\w)\right\|=\left\|a_1\mathbf{c}(\w)_1+\mathbf{c}(\w)_2\right\|$ and prove that
$\left\|a_1\mathbf{c}(\w)_1+\mathbf{c}(\w)_2\right\|\geq 1$. Once this is proved the lemma will be established.
Set $J_1=\{2,i_3,,\ldots,i_j\}$. Then $\mathbf{c}(\w)_1$ has no term containing $\mathbf{e}_{J_1}$
because otherwise, by \eqref{eq: cwdefi}  we will have  $1 \in J_1$. In other words, $\mathbf{c}(\w)_1$ only has terms orthogonal to $\mathbf{e}_{J_1}$.
In addition, $\mathbf{c}(\w)_2=\pm a \mathbf{e}_{J_1}+$  terms orthogonal to $\mathbf{e}_{J_1}$. Hence
\begin{equation}
\left\|a_1\mathbf{c}(\w)_1+\mathbf{c}(\w)_2\right\|\geq \|\pm a \mathbf{e}_{J_1}\|=|a|\geq 1.\qedhere
\end{equation}
\end{proof}

Hence  the assumptions of
Proposition \ref{prop: proposition} are automatically fulfilled for
subgroups $\Gamma$ represented by $\w$ as above,  because from \eqref{eq: critical1} and \eqref{eq: simple} we get
\begin{equation}
  \left\|g_tu_{\mathbf{f}(\cdot)} \w\right\|_{\mu,B}\succ 1 \text{ if } j>1,
  \end{equation}
  where $\succ$ implies some constant dependent on $\mu$, $B$ and $\mathbf{f}$.

  Thus the second assumption of Proposition \ref{prop: proposition} or Condition \ref{condition: equi0}
can be rewritten as:

\begin{condition}\label{condition: equi}
  $\forall k$ $(1\leq k\leq n),\; \forall d_k>c_k$, $\exists T=T(d_k)>0 \text { such that for any }t\geq T \text{ with
    }$  $t_i\geq d_kt$   for at least $k$ values of $i$, $\forall$ nonzero
  $\w \in \Z^{n+1}$ , we have
  \begin{equation}\label{eq:  jdkt}
  \max\left(e^{-t_i}\big\|\langle \mathbf{e}_i,\w \rangle\big\|,\quad e^{t}\big\|R_0\mathbf{c}(\w)\big\|
  \right)\geq e^{-d_kt},\quad 1\leq i\leq n,
  \end{equation}
Matrix $R_0$ is defined via \eqref{eq: rdefi}.
  \end{condition}

In summary, when $\f\colon X\to \LL$ is a continuous map  such that  $(\f,\mu)$ is nonplanar in $\LL$, the second assumption   of Proposition \ref{prop: proposition} $\Leftrightarrow$
Condition \ref{condition: equi0} $\Leftrightarrow$
Condition \ref{condition: equi}.
The next lemma gives an account of what happens if the above conditions fail to hold.

\begin{lem}\label{lem: negation2}
 Let $\mu$ be a Federer measure on a ball $B\subset X$, take $v>n$ and $c_k=\frac{v-n}{kv+n}$ $(1\leq k\leq n)$.
 Let $\mathbf{f}$   be a continuous map from $X$ to $\LL$ such
 that   $(\f,\mu)$ is nonplanar in $\LL$ and the second assumption   of Proposition \ref{prop: proposition} does not hold. Then $\f(B\cap \text{ supp }
 \mu)\subset W_u^\times$ for some $u>v$.
 \end{lem}

 \begin{proof}
 If  the second assumption   of Proposition \ref{prop: proposition} or equivalently Condition \ref{condition: equi} does not hold,
 $\exists k$ with $ 1\leq k \leq n$ , a
 sequence $t^j \rightarrow \infty$ and a sequence of nonzero integer vectors $\w^j$  such that for some
 $d_k>c_k$, we have $\forall \x \in B\cap \text{ supp } \mu$
 \begin{multline}
 \left\|g_{\mathbf{t}^j}u_{\mathbf{f}(\x)}\w^j\right\|\leq e^{-d_kt^j}, \text{ where }
t^j=\sum_{i=1}^n t_i^j\text{ and }t_i^j\geq d_kt^j \text{ for at least } k \text{ values of }i.
 \end{multline}
 Equivalently, $\forall \x\in B\cap \text{ supp }  \mu$,  $\exists m$ independent of $k$
  with $1\leq m \leq n$,  such that
 for an infinite subsequence of $j$,
 there exists nonzero vector $v^j$ such that
 \begin{equation}
 \|v^j\|\leq e^{-d_kt^j},\quad  v^j \in g_{\mathbf{t}^j}u_{\mathbf{f}(\x)}\Z^{n+1}_m.
\end{equation}
Recall that by \eqref{eq: decompose}
\begin{equation}
 \Z^{n+1}_{m}=\left\{\left.(\q,p)=(q_1,\ldots,q_n,p)\in
\Z^{n+1}\right|\text{ exactly $m$ entries of }\q \text{ are nonzero}\right\}.
\end{equation}
Consequently
\begin{equation}
\gamma_m(\mathbf{f}(\x)) \geq d_k
\end{equation}
    We get from \eqref{eq: otyequal} that
  \begin{equation}
  \omega^\times(\f(\x))\geq \dfrac{n+nd_k}{1-md_k}
\end{equation}
If $m \geq k$, then because the function $a(x)= \tfrac{n+nd_k}{1-xd_k}$ increases as $x$ increases, we get
\begin{equation}
  \omega^\times(\f(\x))\geq \dfrac{n+nd_k}{1-md_k}  \geq \dfrac{n+nd_k}{1-kd_k}  >
  \dfrac{n+nc_k}{1-kc_k}=v
\end{equation}
If $m< k$, then the above simple arguments do not apply.
 We have, for an infinite sequence $j$, $\exists (\q^j,p^j) \in \Z_{m}^{n+1}$ such that
 \begin{equation} \label{eq: neg}
\max\left(e^{t^j}|\langle\q^j ,\f(\x)\rangle+p^j|,\quad  e^{-t_i^j}|q_i^j| \right) \leq e^{-d_kt^j}, \quad 1 \leq i \leq n.
\end{equation}
 By assumption $t_i^j \geq  d_kt^j$ for at least $k$ values of  $i$.  For any such $i$, we derive from \eqref{eq: neg}
 \begin{equation}
 |q_i^j|\leq e^{t_i^j-d_kt^j },\quad \text{if  } t_i^j \geq  d_kt^j.
 \end{equation}
From Lemma \ref{lem: elementary} we get that
 \begin{equation}
 |q_i^j|_+\leq e^{t_i^j-d_kt^j },\quad \text{if  } t_i^j \geq  d_kt^j.
 \end{equation}
 Define for each $j$ the following two index sets:
 \begin{equation}
 I_1^j=\{i\mid q_i^j\neq0\},\quad I_2^j=\{i\mid  t_i^j \geq  d_kt^j\}.
 \end{equation}
 By definition
 \begin{equation}
 \Pi_+(\q^j)=\Pi_{i\in I_1^j}|q_i^j|=\Pi_{i\in I_1^j}|q_i^j|_+
 \end{equation}
 Obviously $I_1^j \subset I_2^j$ and this is where the assumption $m<k$ plays a role.
  Hence
 \begin{equation}\label{eq: trick1}
 \Pi_+(\q^j)=\Pi_{i\in I_1^j}|q_i^j|_+\leq \Pi_{i\in I_2^j}|q_i^j|_+.
 \end{equation}
 Now we study $\Pi_{i\in I_2^j}|q_i^j|_+$. Denote by $b$ the number of elements in $I_2^j$. Immediately we get $b\geq k$ from the assumption
 of the lemma that $t_i^j\geq t^j$ for at least $k$ values of $i$. Moreover $b\geq k>m$.
 Hence
 \begin{equation}
 \Pi_{i\in I_1^j}|q_i^j|_+\leq e^{t^j-bd_kt^j}.
 \end{equation}
 Elementary algebra  shows that $e^{t^j-bd_kt^j}\leq e^{t^j-kd_kt^j}$.
 Hence
 \begin{equation}
 \Pi_{i\in I_1^j}|q_i^j|_+\leq e^{t^j-kd_kt^j}.
 \end{equation}
 As a result of the above argument, we have
 \begin{equation}\label{eq: trick}
 \Pi_+(\q^j)\leq e^{t^j-kd_kt^j}.
 \end{equation}
In addition, from \eqref{eq: neg} we have
 \begin{equation}\label{eq: notrick}
 |\langle\q^j, \f(\x)\rangle+p^j|\leq e^{-t^j-d_kt^j}.
 \end{equation}
 From \eqref{eq: trick} and \eqref{eq: notrick} we get    $\omega^\times(\f(\x))\geq \tfrac{n+nd_k}{1-kd_k}  >v$.
 Combining the two cases ($m\geq k$ and $m<k$), we see that  $\omega^\times(\f(\x)\geq \tfrac{n+nd_k}{1-kd_k}>v$,
 $\forall  \x \in B\cap  \text{ supp } \mu$, as desired.
\end{proof}

\begin{theorem}\label{thm: anothermain}
Let $\mu$ be a Federer measure on a Besicovitch metric space $X$, $\LL$ a hyperplane of $\R^n$  and let $\f: X\to \LL$ be a
continuous map  such that  $(\f,\mu)$ is good  and  nonplanar in $\LL$. Then
the following statements are equivalent for $v \geq n$:
\begin{enumerate}
\item $\{\x \in\text{ supp }\mu\mid\f(\x) \notin W^\times_u\}$ is
nonempty for any $u>v$;
\item $\omega^\times(\f_*\mu)\leq v$;
\item Condition \ref{condition: equi0} holds for  $R$ satisfying \eqref{R}, or equivalently,
Condition \ref{condition: equi} holds for  $R_0$ satisfying \eqref{eq: rdefi}.
\end{enumerate}
\end{theorem}

\begin{proof}
Suppose the second statement holds. Then the set in the first statement has full measure and is therefore
 nonempty.

If the third statement holds, then since $\mu$ is Federer and $(\f,\mu)$ is good, we can conclude that  $\mu- a.e.$ $x \in X$
has a neighborhood $V$ such that $\mu$ is $(C,\alpha)$-good and $D$-Federer on $V$ for some $C, D,\alpha>0$.  Choose a ball $B=B(x,r)$ with positive measure such that the dilated ball $\tilde{B}=B(x,3^{n+1}r)$ is contained in $V$.
For any $\w$, each of the coordinates of $g_\mathbf{t}u_\f\w$  is expressed as linear combination of $1,f_1,\ldots, f_n$ from \eqref{eq: uyw2}. By applying  an elementary property, see e.g. \cite[Lemma 4.1]{KLW}, that whenever $f_1,\ldots, f_N$  are $(C,\alpha)$-good on a set $V$ with respect to $\mu$,
the function $(f_{1}^2+\ldots+ f_{N}^2)^{1/2}$ is  $(N^{\alpha/2}C,\alpha)$-good on $V$ with respect to $\mu$, we see that the first assumption of Proposition \ref{prop: proposition} is satisfied. The second assumption can be derived from Condition \ref{condition: equi} by previous discussion concerning the nonplanarity in $\LL$.  Hence we can apply Proposition \ref{prop: proposition}
to establish the second statement.

If the third statement fails to hold, then no ball $B$ intersecting  $\text{ supp }\mu$
satisfies Condition \ref{condition: equi}. By Lemma \ref{lem: negation2} $\f(B\cap \text{ supp }\mu) \subset W^\times_u$ for some $u>v$.
This  contradicts the first statement.
\end{proof}

From Theorem \ref{thm: anothermain} we see that $\omega^\times(\LL)\leq\inf\{\omega^\times(\y)\mid \y \in \LL\}$ as the
first statement of the theorem implies the second one.
$\omega^\times(\LL)\geq\inf\{\omega^\times(\y)\mid \y \in \LL\}$
 can be derived from  definition.
$\omega^\times(\LL)$ is inherited by
its nondegenerate submanifolds as  nondegeneracy in $\LL$ implies  nonplanarity  in  $\LL$ by definition.
 Therefore \[\omega^\times(\LL)=\omega^\times(\MM)=\inf\{\omega^\times(\y)\mid \y \in \LL\}=\inf\{\omega^\times(\y)\mid \y \in \MM\}\]
  and Theorem \ref{thm: mainthm} is established.

Besides, Theorem \ref{thm: anothermain} establishes that
\begin{equation}\label{eq: computingsigmal}
\omega^\times(\LL)=\sup\left\{v\mid \text{ Condition }\ref{condition: equi}\text{ does not hold}\right\}.
\end{equation}
For hyperplane $\LL$ defined in Theorem  \ref{thm: oml}, we embed it into $\R^{n+1}$ as \ref{eq: rdefi1} by
\begin{equation}\label{eq: fdefi}
\widetilde{\f}(\x)=(a_1x_1+\ldots+a_{n-1}x_{n-1}+a_n,x_1,\ldots,x_{n-1},1).
\end{equation}
Now we prove Theorem \ref{thm: oml}. Without loss of generality, we suppose from now on that
 \begin{equation}\label{arestrict}
a_1a_2\ldots a_{s-1}\neq 0, \quad \text{ and } a_i=0 \text{  for  }  s\leq i\leq n-1.
\end{equation}
Suppose $\w=(p_1,\ldots,p_n,p_0) \in \Z^{n+1}$, then since $j=1$ the index set $J$ with order $1-1=0$  becomes empty and $\bigwedge^{j-1}(\R^{n+1})\in \Z$, we have
 \begin{equation}
 \mathbf{c}(\w)_i=\pm\left\langle \mathbf{e}_i,\w \right\rangle=\pm p_i\;(1\leq i\leq n),\quad \mathbf{c}(\w)_{n+1}=\pm\left\langle \mathbf{e}_{n+1},\w \right\rangle=\pm p_0.
 \end{equation}
 We can change the signs of $p_i$, so we will just use + instead of $\pm$ from now on. Note that
 \begin{equation}
 \mathbf{c}(\w)=\left(  \begin{array}{c}
 p_1\\ p_2\\ \vdots \\ p_n \\  p_{0}\\ \end{array}
 \right)
 \end{equation}
 and
 \begin{equation}
 R_0\mathbf{c}(\w)=\left(  \begin{array}{c}
a_1\\a_2\\ \vdots \\  a_{n-1} \\a_n \\ \end{array}  I_n\right)\left(  \begin{array}{c}
 p_1\\ p_2\\ \vdots \\ p_n \\  p_{0}\\ \end{array}
 \right).
 \end{equation}
 Therefore
\begin{equation}\label{eq: rcj}
 \big\|R_0\mathbf{c}(\w)\big\|=\left\|  \begin{array}{c}
 a_1p_1+p_2\\ \vdots \\a_{s-1}p_1+p_s\\a_{s}p_1+p_{s+1}\\ \vdots \\ a_{n-1}p_1+p_n \\  a_np_1+p_{0}\\ \end{array}
 \right\|.
 \end{equation}
 Unless  $p_{s+1}=\ldots=p_{n}=0$ and $p_1\ldots p_s \neq 0$, $\big\|R_0\mathbf{c}(\w)\big\|\geq \epsilon$  for some positive fixed number $\epsilon$
 whenever $\w$ is nonzero.
 In other words  the second assumption of Proposition \ref{prop: proposition} is always satisfied except for
 $\w \in \Z_s^{n+1}$.
The above observations coupled with Proposition \ref{prop: proposition} supply a useful tool for establishing upper
 bounds of multiplicative exponents of hyperplanes described in \eqref{eq: fdefi}.
 Proof of Theorem \ref{thm: oml} is based on \eqref{eq: computingsigmal}:

 \begin{proof}[Proof of Theorem \ref{thm: oml}]
We employ the method of proof of Lemma \ref{lem: negation2}.

If Condition \ref{condition: equi} does not hold,  $  \exists k$ $(1\leq k\leq n$, $k$ independent of $s$) such that for some $d_k>c_k$,
  $\exists$ an unbounded sequence of $t$ with  $t_i \geq d_kt \text{ for at least }k\text{ values of }i$ and a sequence of  $ \w \in \Z^{n+1}_s$,
    one has
\begin{equation}\label{eq: jdkt3}
 \max\left(e^{-t_i}|p_i|,\quad e^t\left\|R_0\mathbf{c}(\w)\right\|\right)\leq e^{-d_kt},\quad 1\leq i\leq n
 \end{equation}
 $\|R_0\mathbf{c}(\w)\|$ is defined in \eqref{eq: rcj}.
 After reordering, we may assume that $t_i \geq d_kt$ when $1\leq i \leq k$.
 Consequently we have
 \begin{equation}
 |p_i|_+ \leq e^{t_i-d_kt},\qquad 1\leq i \leq k,
 \end{equation}
\begin{equation}
\Pi_+(\mathbf{p})\leq |p_1|_+\ldots |p_k|_+< e^{t-kd_kt},
\end{equation}
\begin{equation}
\left\|R_0\mathbf{c}(\w)\right\|\leq e^{t-d_kt}.
\end{equation}
Hence for some $u>v$
\begin{equation}
\left\|R_0\mathbf{c}(\w)\right\|<\Pi_+(\mathbf{p})^{-u/n}.
\end{equation}
Note that on the other hand by our assumption
 \begin{equation}
\Pi_+(\mathbf{p})=|p_1p_2\ldots p_s|.
 \end{equation}
Hence \eqref{eq: jdkt3} is equivalent to: $\exists$ an infinite
  sequence of
  $(p_1,p_2,\ldots,p_s,p_0) \in \Z^{s+1}_s$  such that for some $u>v$ we have
 \begin{equation}\label{eq: jdkt4}
 \left\|  \begin{array}{c}
 a_1p_1+p_2\\a_2p_1+p_3 \\\vdots\\ a_{s-1}p_1+p_s\\a_np_1+p_0\\ \end{array}  \right\|
  < |p_1p_2\ldots p_s|^{-u/n}.
 \end{equation}
 By assuming $\|p_{i+1}+a_{i}p_1 \|\leq 1$ for $1\leq i\leq s-1$, we deduce that
 $|p_i|\asymp  |p_1|$ for  $1 \leq i \leq s$.
Thus \eqref{eq: jdkt4} is equivalent to : $\exists$ a
  sequence of $(p_1,p_2,\ldots,p_s,p_0) \in \Z^{s+1}_s$ with $|p_1|$ unbounded such that for some $u>v$
 \begin{equation} \label{eq: jdkt5}
 \left\|  \begin{array}{c}
 a_1p_1+p_2\\ a_2p_1+p_3 \\ \vdots \\a_{s-1}p_1+ p_{s} \\ a_np_1+p_0
 \end{array}  \right\|
  \prec|p_1|^{-su/n},
 \end{equation}
 where $\prec$ implies some constant dependent on $\mathbf{a}$.

 According to \eqref{eq: sigma}, $\sigma(\mathbf{a})$ is exactly $\tfrac{s}{n}\sup\left\{v\mid \eqref{eq: jdkt5}\text{ holds}\right\}$.
Therefore by \eqref{eq: computingsigmal}  $\omega^\times(\LL)=\max\left(n,\dfrac{n}{s}\sigma(\mathbf{a})\right)$.
 Theorem \ref{thm: oml} is proved.
 \end{proof}

\section{A Special Case}

We consider a special class of  hyperplanes whose multiplicative Diophantine exponents can be obtained in an
elementary manner:

\begin{theorem}
Let $\LL$ be a hyperplane in $\R^n$ parameterized by
\begin{equation}
\LL=\left\{(x_1,x_2,\ldots,x_{n-1},a)\left|(x_1,\ldots,x_{n-1})\in \R^{n-1}\right.\right\}
\end{equation}
Then $\omega^\times(\LL)=n\sigma(a)$.
\end{theorem}

This is a special case of Theorem \ref{thm: oml} with $s=1$ and $a_i=0$ for $1\leq i \leq n-1$.

\begin{proof}
For an arbitrary $\y=(x_1,x_2,\ldots,x_{n-1},a) \in \LL$, if we approximate it by $\q \in \Z^n$ of the special form
$(0,\ldots,0,q_n)$, we see from \eqref{omegatimes} that
\begin{equation}\label{eq: lgreater}
\ot\geq n\sigma(a),\quad \forall \y \in \LL.
\end{equation}
Hence $\omega^\times(\LL)\geq n\sigma(a)$. We proceed to prove that $\omega^\times(\LL)\leq n\sigma(a)$.
Apparently,
\begin{equation}
 \Pi_+(\q) \geq \|\q\|,\quad \forall \q \in \Z^n,
\end{equation}
hence from \eqref{eq: omega} and \eqref{omegatimes} we get
\begin{equation}
\ot\leq n\omega(\y)\quad \forall \y \in \R^n.
\end{equation}
On the other hand it is known from \cite{Jarnik} that  $\omega(\y)=\sigma(a),\; a.e.\;\y \in \LL$.
Hence
\begin{equation}\label{eq: lsmaller}
\ot\leq n\omega(\y)=n\sigma(a),\; a.e.\;\y \in \LL
\end{equation}
Combining \eqref{eq: lgreater} and \eqref{eq: lsmaller} we have $\omega^\times(\LL)=n\sigma(a)$.
\end{proof}

\begin{remark}
It is still an open question whether Theorem \ref{thm: mainthm} holds for subspaces of  codimension bigger than $1$. 
\end{remark}

\end{document}